\documentclass{agtart_a}
\pdfoutput=1

\usepackage[all]{xy}

\title{Categorical sequences}

\author{Rob Nendorf}
\givenname{Rob}
\surname{Nendorf}
\address{Department of Mathematics\\
Northwestern University\\\newline
Evanston, IL 60208-2730\\
USA}
\email{rnendorf@math.northwestern.edu}
\urladdr{}

\author{Nick Scoville}
\givenname{Nick}
\surname{Scoville}
\address{Department of Mathematics \\
Dartmouth College \\\newline
Hanover, NH 03755-3551\\
USA}
\email{nicholas.scoville@dartmouth.edu}
\urladdr{}

\author{Jeff Strom}
\givenname{Jeffrey}
\surname{Strom}
\address{Department of Mathematics \\
Western Michigan University \\\newline
Kalamazoo, MI 49008\\
USA}
\email{jeff.strom@wmich.edu}
\urladdr{}

\volumenumber{6}
\issuenumber{}
\publicationyear{2006}
\papernumber{30}
\startpage{809}
\endpage{838}

\doi{}
\MR{}
\Zbl{}



\received{5 January 2006}
\revised{}
\accepted{23 April 2006}
\published{21 June 2006}
\publishedonline{21 June 2006}
\proposed{}
\seconded{}
\corresponding{}
\editor{}
\version{}

\arxivreference{}  

\let\xysavmatrix\xymatrix
\def\xymatrix{\disablesubscriptcorrection\xysavmatrix}
\AtBeginDocument{\let\bar\wbar\let\tilde\wtilde\let\hat\what}

\makeatletter
\def\cnewtheorem#1[#2]#3{\newtheorem{#1}{#3}[section]
\expandafter\let\csname c@#1\endcsname\c@thm}

\newtheorem{thm}{Theorem}[section]

\cnewtheorem{lem}[thm]   {Lemma}
\cnewtheorem{cor}[thm]   {Corollary}
\cnewtheorem{prop}[thm]  {Proposition}

\newcounter{foo}  
\newtheorem{conj}[foo]{Conjecture}

\theoremstyle{remark}
\cnewtheorem{rem}[thm]   {Remark}
\cnewtheorem{defn}[thm]  {Definition}
\cnewtheorem{authex}[thm]    {Example}

\makeautorefname{foo}{Conjecture}
\makeautorefname{authex}{Example}
\makeautorefname{defn}{Definition}

\makeatother  

\newcommand{\term}[1]   {\emph{#1}}

\newcommand{\inclds}     {\hookrightarrow}

\newcommand{\A}        {\mathcal{A}}
\newcommand{\B}        {\mathcal{B}}
\newcommand{\G}        {\mathcal{G}}
\newcommand{\M}        {\mathcal{M}}
\newcommand{\N}        {\mathcal{N}}

\newcommand{\ZZ}    {\mathbb{Z}}
\newcommand{\RR}    {\mathbb{R}}
\newcommand{\CC}    {\mathbb{C}}
\newcommand{\NN}    {\mathbb{N}}
\newcommand{\QQ}    {\mathbb{Q}}

\newcommand{\s}     {\Sigma}
\newcommand{\smsh}  {\wedge}
\newcommand{\om}    {\Omega}
\newcommand{\cross} {\times}
\newcommand{\wdg}   {\vee}
\newcommand{\of}    {\circ}
\newcommand{\id}    {\mathrm{id}}
\newcommand{\sseq}  {\subseteq}
\newcommand{\tensor}   {\otimes}

\newcommand{\Hom}   {\mathrm{Hom}}
\newcommand{\Ext}   {\mathrm{Ext}}
\newcommand{\hocolim}{\mathrm{hocolim}}

\newcommand{\wcat}  {\mathrm{wcat}}

\newcommand{\cat}   {\mathrm{cat}}
\newcommand{\nil}   {\mathrm{nil}}
\newcommand{\Mcat}   {\mathrm{Mcat}} 
\newcommand{\Qcat}   {\mathrm{Qcat}}
\newcommand{\wgt}   {\mathrm{wgt}}

\newcommand{\cp}    {\CC\mathrm{P}}

\newcommand{\map}   {\mathrm{map}}

\newenvironment{introthm}[1]
{\medskip \noindent \textbf{#1}    \it}
{\rm\smallskip}

\newcommand{\mmathand}{\qquad\mathrm{and}\qquad}

\newenvironment{problist}
{\begin{enumerate}}
{\end{enumerate}}

\begin{document}

\begin{asciiabstract}
We define and study the categorical sequence of a space, which is a new
formalism that streamlines the computation  of the Lusternik-Schnirelmann
category of a space X by induction on its CW skeleta.  The k-th term in
the categorical sequence of a CW complex X, \sigma_X(k), is the least
integer n for which cat_X(X_n) >= k.  We show that \sigma_X is a
well-defined homotopy invariant of X.

We prove that \sigma_X(k+l) >= \sigma_X(k) + \sigma_X(l),
which is one of three keys to the power of categorical sequences.
In addition to this formula, we provide formulas relating the categorical
sequences of spaces and some of their algebraic invariants, including
their cohomology algebras and their rational models; we  also find
relations between the categorical sequences of the spaces in a fibration
sequence and give a preliminary result on the categorical sequence of a
product of two spaces in the rational case.  We completely characterize
the sequences which can arise as categorical sequences of formal rational
spaces.  The most important of the many examples that we offer is a simple
proof of a theorem of Ghienne:  if X is a member of the Mislin genus of
the Lie group Sp(3), then cat(X) = cat(Sp(3)).
\end{asciiabstract}

\begin{htmlabstract} 
We define and study the categorical sequence of a space, which is a new
formalism that streamlines the computation  of the Lusternik&ndash;Schnirelmann
category of a space X by induction on its CW skeleta.  The
k<sup>th</sup> term in the categorical sequence of a CW complex X,
&sigma;<sub>X</sub>(k), is the least integer n for which cat<sub>X</sub>(X<sub>n</sub>)&ge;k.
We show  that &sigma;<sub>X</sub> is a well-defined homotopy invariant of X.
We prove that &sigma;<sub>X</sub>(k+l)&ge;&sigma;<sub>X</sub>(k)+&sigma;<sub>X</sub>(l), which is
one of three keys to the power of categorical sequences.  In addition to
this formula, we provide formulas relating the categorical sequences of
spaces and some of their algebraic invariants, including their cohomology
algebras and their rational models; we  also find relations between the
categorical sequences of the spaces in a fibration sequence and   give
a preliminary result on the categorical sequence of a product of two
spaces in the rational case.  We completely characterize the sequences
which can arise as categorical sequences of formal rational spaces.
The most important of the many examples that we offer is a simple proof
of a theorem of Ghienne:  if X is a member of the Mislin genus of the
Lie group Sp(3), then cat(X) = cat(Sp(3)).
\end{htmlabstract}

\begin{abstract} 
We define and study the categorical sequence of a space, which is a new
formalism that streamlines the computation  of the Lusternik--Schnirelmann
category of a space $X$ by induction on its CW skeleta.  The
$k^{\mathrm{th}}$ term in the categorical sequence of a CW complex $X$,
$\sigma_X(k)$, is the least integer $n$ for which $\mathrm{cat}_X(X_n) \geq k$.
We show  that $\sigma_X$ is a well-defined homotopy invariant of $X$.
We prove that $\sigma_X(k+l) \geq \sigma_X(k) + \sigma_X(l)$, which is
one of three keys to the power of categorical sequences.  In addition to
this formula, we provide formulas relating the categorical sequences of
spaces and some of their algebraic invariants, including their cohomology
algebras and their rational models; we  also find relations between the
categorical sequences of the spaces in a fibration sequence and   give
a preliminary result on the categorical sequence of a product of two
spaces in the rational case.  We completely characterize the sequences
which can arise as categorical sequences of formal rational spaces.
The most important of the many examples that we offer is a simple proof
of a theorem of Ghienne:  if $X$ is a member of the Mislin genus of the
Lie group $Sp(3)$, then $\mathrm{cat}(X) = \mathrm{cat}(Sp(3))$.
\end{abstract}

\maketitle


\section*{Introduction}
 
The \term{Lusternik--Schnirelmann category} of a topological space
$X$ is the least integer $k$ for which $X$ has an open cover
$X = X_0 \cup X_1 \cup \cdots \cup X_k$ with the property that
each inclusion map $X_j \inclds X$ is homotopic to a constant map;
it is denoted $\cat(X)$.   
This   homotopy invariant of topological spaces was first introduced
by Lusternik and Schnirelmann in 1934 as a tool to use in studying 
functions on (compact) manifolds:  a smooth function $f \co M \to \RR$
must have at least $\cat(M) + 1$ critical points.

If $X$ is a CW complex,  then 
$X_{n} = X_{n-1} \cup_\alpha (n\mathrm{-cells})$, and therefore
$\cat(X_n) \leq \cat(X_{n-1}) + 1$.  Berstein and Hilton asked \cite{B-H}
what conditions must be placed on the attaching map $\alpha$ in 
order to guarantee that equality holds in this upper bound;
the answer is  that equality holds   when 
a certain set of generalized Hopf invariants does not contain the 
trivial map $*$.   Thus it is possible, at least in principle,
to compute the Lusternik--Schnirelmann category of a 
finite-dimensional CW complex inductively up its skeleta.

It was shown by the third author \cite{Strom} that the Hopf sets
for lower-dimensional skeleta partially determine the Hopf sets for
high-dimensional skeleta.  In actual computations, this makes it possible
to `bootstrap' up from relatively simple low-dimensional results to
(apparently) difficult high-dimensional calculations.  Our goal in
this paper is to establish a convenient formalism for doing category
calculations while making use of all low-dimensional information.

This is done via the \term{categorical sequence} of a space $X$,
which is a function $\sigma_X \co \NN\to \NN\cup\{\infty\}$ defined by 
$$
\sigma_X(k) = \inf\{  n \, |\,  \cat_X(X_n) \geq k\} ,
$$
where $\cat_X(X_n)$ is the category of $X_n$ relative to  $X$
(see Definition 4)\footnote{In his encyclopedic paper \cite{Fox}, Fox
defined a \textit{categorical sequence} to be a certain kind of
filtration of a space;  this idea was used in his proof of the 
product inequality $\cat(X\cross Y) \leq \cat(X) + \cat(Y)$.  
Our use of the term
 `categorical sequence' 
for a completely different idea should carry no risk of confusion,
since the earlier notion is no longer used, at least in the
homotopy theory of Lusternik--Schnirelmann category (but see
Cicorta\c{s} \cite{Dude1,Dude2} for an equivariant version).}.
It is shown in Propositions \ref{prop:invariant} and  \ref{prop:WellDefined} 
that $\sigma_X$ is a well-defined
homotopy invariant of $X$; ie,  when $n$ is larger than
the connectivity of  $X$,  $\cat_X(X_n)$ depends only 
on $n$ and the homotopy type of $X$, and not on any choices made in 
constructing a CW decomposition of $X$.  
If $X$ is finite-dimensional, then $\sigma_X$ determines $\cat(X)$;
examples due to Roitberg \cite{Roitberg} show that this is not 
true for infinite-dimensional spaces.
In any case, the categorical sequence of $X$ holds a wealth of useful 
information.      

Though we are not directly concerned with the
applications of Lus\-ter\-nik\---\-Schnir\-el\-mann
category to critical point theory in this paper, 
categorical sequences could play a useful role there.
For example, in the study of the $n$--body problem,  
one is often interested in infinite-dimensional Sobolev 
spaces $W$; 
in order to apply the Lusternik--Schnirelmann method in this situation,
it is necessary to find compact subsets $K\sseq W$ such that the  
relative category $\cat_W(K)$ is large 
(see Ambrosetti and Zelati \cite[Remarks 2.15 and 3.5]{A-Z}, Fadell and
Husseini \cite[Theorem 4.6]{Fa-Hu}, or Rabinowitz \cite{Rabinowitz},  
for example).   
The categorical sequence of $W$ gives \textit{lower}
estimates on the dimension of such subsets.  If $\sigma_W(k) = n$, then 
$\cat_W(W_{n-1}) < k$;  if $\dim(K) < n$ then $K$ can be deformed 
into   $W_{n-1}$, and so $\cat_W(K) <   k$. 

Our theoretical results establish formulas for calculation with 
categorical sequences.  Some of the statements make use of 
another sequence, the \term{product length sequence} of a 
non-negatively graded
commutative  algebra $\A$, defined by setting $\sigma_\A(k)$
to be the least dimension $n$ for which 
the $n^\mathrm{th}$ grading $\A^n$ of $\A$
contains a nontrivial 
$k$--fold product.  

\begin{introthm}{\fullref{prop:cuplength}}
For any space $X$ and any  ring $R$, 
$\sigma_X \leq \sigma_{H^*(X;R)}$.
\end{introthm}

\noindent
We also   estimate the categorical sequence
of a rational space in terms of any of its models.

\begin{introthm}{\fullref{prop:model}}
If $X$ is a simply-connected rational space and $\A$ is
any model   for  $X$, then
$\sigma_X \geq \sigma_\A$.
\end{introthm}
  
 \noindent 
Recall that a simply-connected rational space $X$ is \textit{formal}
if its cohomology algebra $H^*(X)$, with trivial differential, is a model for $X$.
Thus we have the following computation for formal spaces.

\begin{introthm}{\fullref{prop:formal}}
If $X$ is a simply-connected formal rational space, then
 $\sigma_X = \sigma_{H^*(X)}$.
\end{introthm}

\noindent
More generally, we completely determine the sequences $\sigma$
which can arise as the categorical sequences of 
simply-connected formal rational spaces.

\begin{introthm}{\fullref{thm:formal}}
The following conditions on a sequence $\sigma$ with $\sigma(1)>1$
are equivalent:
\begin{problist}
\item 
$\sigma = \sigma_\A$ for some commutative graded algebra $\A$,
\item  
$\sigma(k+1) \geq  {k+1\over k}\thinspace \sigma(k)$ for each $k$,  
\item  
$\sigma = \sigma_W$ where 
 $W= \bigvee P_i$ and   $P_i = \prod S^{n_j}$ 
is a product of spheres, and 
\item 
$\sigma= \sigma_X$ for some formal rational space $X$.
\end{problist}
\end{introthm}

The keys to the computational power of categorical sequences, though, 
are the three properties listed in the following theorem.   
In order to prove parts (b) and (c) for \textit{all} spaces
(and not just spaces of finite type, say), we have to 
use a set-theoretical framework in which Whitehead's Problem
(which asks: does $\Ext(A,\ZZ) =0$ imply $A$ is free?) has a positive solution.
See \fullref{sec1.1} and \fullref{lem:homsection} for   details.

\begin{introthm}{\fullref{thm:subadditive}}
For any space $X$,
\begin{problist}
\item  $\sigma_X(k+l) \geq \sigma_X (k) + \sigma_X(l)$, 
\item  if $X$ is simply-connected  and $\sigma_X(k) = n$, then 
$H^n(X;A)\neq 0$ for some coefficient group $A$, and 
\item  if equality occurs in (a) and $X$ is simply-connected, 
then the cup product 
$$
H^k(X;A) \otimes H^l(X;B) \to H^{k+l}(X;A\otimes B)
$$
is nontrivial for some choice of coefficients.
\end{problist}
\end{introthm}

The point we hope to make in this paper
is that calculation with sequences is no harder than
calculation of category;  indeed, the extra information contained
in the sequence, together with \fullref{thm:subadditive},
 can greatly facilitate computations.  
To illustrate this point, 
let $X$ be any  simply-connected space such that 
$H^*(X;\ZZ) \cong H^*(Sp(3);\ZZ) = \Lambda(x_3,x_7,x_{11})$. 
The categorical
sequence $\sigma_X$ clearly  has $\sigma_X(1) = 3$ and 
$\sigma_X(2) \geq 7$ by \fullref{thm:subadditive}(b).  
By \fullref{thm:subadditive}(a),
$\sigma_X(4)\geq \sigma_X(2) + \sigma_X(2) \geq 14$.
Furthermore, $\sigma_X(4) > 14$ by \fullref{thm:subadditive}(c),
because
the cup product $H^7(X) \tensor H^7(X) \to H^{14}(X)$ is trivial.
Now we have $\sigma_X(4) \geq 18$ by \fullref{thm:subadditive}(b),
and hence  $\sigma_X(5) \geq \sigma_X(4) + \sigma_X(1) = 21$
by \fullref{thm:subadditive}(a).
Since $X$ is $21$--dimensional (up to homotopy), this proves 
that $\cat(X) \leq 5$ by \fullref{thm:subadditive}(b).
We will see in \fullref{thm:genus} below that this 
calculation constitutes a simple proof of a result of Ghienne \cite{Ghienne}
about the Mislin genus of $Sp(3)$. 

 \fullref{thm:subadditive} can also be used to prove a generalization
of a somewhat obscure result of Ganea \cite{Ganea}.

\begin{introthm}{\fullref{cor:Ganea}}
Let $X$ be  simply-connected  and of finite type with
 $\sigma_X(k) = n$.  If  there are  integers
$0<a_1< a_2<  \cdots <   a_l$ such that 
$$
 \{ n \,  |\, \wwtilde{H}^n(X;G) \neq 0\ \mathrm{for\ some}\ G \} \sseq  
I_1 \cup 
I_2 \cup \cdots \cup
I_l  
$$
where $I_j = [a_j, a_j + (n-1)]$
(brackets denote closed  intervals in $\RR$), 
then $\cat(X) < k (l+1)$.   
\end{introthm}
 
The importance
of \fullref{cor:Ganea}
is not the result as such.  Rather, it is the fact that, since
it simply encodes an elementary computation with sequences,
the result can be safely disregarded without losing computational power.
Our proof is \textit{completely different} from the one given in 
\cite{Ganea}.
Ganea's proof makes use of the Blakers--Massey theorem:
certain cofiber sequences are treated as fibration sequences.  Our argument uses 
\fullref{thm:subadditive}, which in turn rests on a much more
elementary fact:  the
factorization $\Delta_{k+l} = (\Delta_k\cross \Delta_l) \of \Delta_2$
of  diagonal maps.
However, Ganea's theorem also applies to the strong category of $X$,
while ours only works for ordinary category.

One of our most pleasing general results gives formulas
relating the categorical sequences of the spaces in a fibration
sequence. 

\begin{introthm}{\fullref{thm:fibseq}}
Let  
$
\xymatrix@1{
F\ar[r]^q & E \ar[r]^p & B
}
$
be a fibration sequence and write $a= \cat(q) \leq \cat(F)$
and $b = \cat(p) \leq \cat(B)$.  Then
\begin{problist}
\item
$\sigma_E(k(a+1)  ) \geq \sigma_B(k)$, and 
\item
$\sigma_E(k(b+1)) \geq \sigma_F(k)$.
\end{problist}
\end{introthm}

\noindent
As a corollary to \fullref{thm:fibseq} we obtain the following
elaboration of the celebrated Mapping Theorem from the rational
theory of Lusternik--Schnirelmann category.

\begin{introthm}{\fullref{prop:mapping}}
Let $f:X\to Y$ be a map between simply-connected rational spaces which 
induces an injective map $f_*: \pi_*(X)   \to \pi_*(Y)$.
Then 
$\sigma_X \geq \sigma_Y$.
\end{introthm}

Finally, we address products.   To state our result  (and  our conjectures),
we construct, for sequences  $\sigma$ and $\tau$, 
a `product sequence' $\sigma*\tau$   defined by 
$
\sigma*\tau (k) = \min \{ \sigma(i) + \tau(j) \, |\, i+ j = k\}.
$
It is not hard to see, using \fullref{prop:model},
 that if $X$ and $Y$ are  simply-connected formal rational spaces, 
then $\sigma_{X\cross Y } = \sigma_X * \sigma_Y$.  We conjecture
that this equation holds in general for simply-connected rational spaces.
So far however, the best we have been able to do is an inequality.

\begin{introthm}{\fullref{thm:product}}
For  simply-connected rational   $X$ and $Y$,  
$
\sigma_{X\cross Y} \leq  \sigma_X * \sigma_Y
$.
\end{introthm}

\noindent
This inequality is certainly not true in general for non-rational spaces,
as the examples of Iwase \cite{Iwase} show.  However, we conjecture
that the \textit{reverse} inequality 
$\sigma_X * \sigma_Y \leq \sigma_{X\cross Y}$
is valid,  not only   for rational spaces,
but for \textit{all} spaces. 

\subsubsection*{Acknowledgements}
We thank Don Stanley, Martin Arkowitz, Yves F\'elix, Daniel
Tanr\'e  and Terrell Hodge for their interest and advice at various
stages of this project.

\section{Preliminaries}

In this section we establish the basic notation and 
concepts that will be used in the body of the paper.

\subsection{Basics}
\label{sec1.1}

We work with pointed spaces and maps;  we use $*$
to denote the one point space or any trivial map.  
We use $\id_X \co X\to X$ to denote the identity map
and $\Delta_k \co X\to X^k$ to denote the diagonal map 
$\Delta_k(x) = (x,x,\ldots, x)$.
The symbol $\simeq$ denotes homotopy equivalence of
spaces or homotopy of maps.   All solid arrow diagrams in this
paper are (homotopy) commutative.  

If $S$ is a set of real numbers, then $\inf(S)$ is the 
infimum of  $S$.  We adopt the usual
convention that $\inf(\varnothing) = \infty$.

\subsubsection*{Set theory}
Whitehead's Problem asks:
if $A$ is an abelian group such that  $\Ext(A,\ZZ) = 0$, does it follow that $A$ is free? 
The answer is `yes' if $A$ is finitely generated.
Shelah has shown that  the general problem is undecidable
in ordinary ZFC set theory, but  the answer is `yes' if G\"odel's
Axiom of Constructibility is assumed (see Shelah \cite{Shelah}). 
In order to avoid `unnecessary' hypotheses in 
\fullref{lem:homsection} and \fullref{thm:subadditive}
below,  
\begin{enumerate}
\item[$\star$]
 we  work in a set theory where
$\Ext(A,\ZZ) = 0$ implies that $A$ is   free.
\end{enumerate}
For those uncomfortable with this assumption, 
we emphasize that {ordinary ZFC set theory is sufficient
to prove \fullref{lem:homsection} and \fullref{thm:subadditive}
when  $H^n(X;\ZZ)$ is finitely generated for each $n$.}

\subsection{Skeleta}

We are concerned with the Lusternik--Schnirelmann 
category of the CW skeleta of a space.  It will simplify 
some of our later work to use the following
slightly abstract notion of skeleton.

\begin{defn}
An \term{$n$--skeleton} for a space $X$ is a  map
$i \co X_n\to X$, where $X_n$ is a CW complex such that  
\begin{problist}
\item  $X_n$ is $n$--dimensional (up to homotopy), and 
\item  $i$ is an $n$--equivalence.
\end{problist}
\end{defn}

\noindent
This definition is justified by the observation that
an $n$--skeleton $i \co X_n\to X$ can be taken as the 
$n^{\mathrm{th}}$ CW skeleton of a CW replacement 
for $X$.

The following result will help us to recognize skeleta.
We omit the proof.

\begin{lem}\label{lem:skeleta}
Let $i \co A\to X$ where $A$ and $X$ are simply-connected.
Then $i$ is an $n$--skeleton for $X$ if and only if 
\begin{problist}
\item
$H^*(A) = 0$ for $*>n$ in all coefficients, and 
\item
the induced map $i^* \co H^*(X) \to H^*(A)$ 
is an isomorphism for $*< n$ and is injective for $*= n$
in all coefficients.
\end{problist}
\end{lem}

When we are working with rational spaces (see F\'elix, Halperin and Thomas
\cite{FHT0})  we will want our
skeleta to also be rational spaces.  Unfortunately, this won't
always happen;  for example, the inclusion
$\bigvee_{n=1}^\infty  S^n \inclds  S^n_{\QQ}$
is an $n$--skeleton for the rational $n$--sphere.  We avoid
this problem by defining a \term{rational $n$--skeleton} of
a simply-connected rational space $X$ to be a map $i \co X_n \to X$ where
$X_n$ is a simply-connected rational space such that
\begin{problist}
\item
$H^*(X_n;\QQ) = 0$ for $*>n$, and 
\item
the induced map $i^* \co H^*(X;\QQ) \to H^*(X_n;\QQ)$ 
is an isomorphism for $*< n$ and is injective for $*= n$.
\end{problist}
Rational $n$--skeleta are plentiful:
if $X$ is a rational space and $X_n$ is an (integral)
$n$--skeleton of $X$, then $(X_n)_\QQ$ is a rational $n$--skeleton of $X$.
We make a standing convention
that if a space $X$ is assumed to be rational, then whenever we
refer to an $n$--skeleton of $X$,  we actually mean a
\textit{rational}  $n$--skeleton.

The proof of part (b) of our next result in full generality
depends on a positive solution to Whitehead's Problem.

\begin{lem}\label{lem:homsection}
Let $X$ be a simply-connected space.
\begin{problist}
\item
If $H_n(X;\ZZ)$ is free abelian, then $X$ has an $n$--skeleton 
$i \co X_n\to X$ such that $i^* \co H^*(X) \to H^*(X_n)$ is an isomorphism
for $*\leq n$,
\item 
If $H^n(X;A) = 0$ for all coefficient groups $A$, then $X$
has an $(n{-}1)$--dimensional $n$--skeleton.
\end{problist}
The corresponding statements also hold for all simply-connected 
rational spaces. 
\end{lem}

 \begin{proof}
Write $M(G,n)$ for the Moore space with
$$H_n(M(G,n);\ZZ) \cong G
  \quad\text{and}\quad
  H_i(M(G,n);\ZZ) = 0 \quad\text{for } i\neq 0,n.$$
When $G$ is free abelian, we   take $M(G,n) = \bigvee S^n$.
According to Brown and Copeland \cite{BrownCopeland},
any simply-connected  space $X$ admits a homology decomposition,
ie, a sequence of CW complexes  $X(n)$ which are related to one another
by cofiber sequences $M_{n-1} \to X(n-1) \to X(n)$
(where $M_{n-1} = M(H_n(X;\ZZ), n-1)$)
and satisfy $X\simeq \hocolim_n \, X(n)$.
The inclusion map $X(n) \to X$ induces isomorphisms on
integral homology
through dimension $n$, and $H_k(X(n);\ZZ) = 0$ for $k>n$.

With the CW decomposition inherited from the colimit, 
$X_n \sseq X(n) \sseq X_{n+1}$ for each $n$.
If $H_n(X;\ZZ)$ is free abelian
then $X(n) = X(n-1) \cup (n \mathrm{-cells})$,
so $X(n)$ is $n$--dimensional and hence $X_n = X(n)$.
Thus $X(n)$ is the desired $n$--skeleton of $X$.

To prove (b), assume that $H^n(X;A) = 0$ for all $A$.
Using the Universal Coefficient isomorphism 
(see Switzer \cite[Corollary 13.11]{Switzer}), we obtain
$$
\Hom(H_n(X;\ZZ), A) = 0 \mmathand \Ext(H_{n-1}(X;\ZZ);A ) = 0 
$$
for all $A$.
We claim that (i) $H_n(X;\ZZ) = 0$ and (ii) $H_{n-1}(X;\ZZ)$ is free
abelian.   To prove (i), we let $A = H_n(X;\ZZ)$; if $A$
were nonzero, then $\Hom ( H_n(X;\ZZ),A)$ would be nonzero 
(since it contains the identity map), thereby contradicting the assumption.
For (ii), we set $A = \ZZ$;  now
$\Ext(H_{n-1}(X;\ZZ),\ZZ) = 0$, and by Whitehead's Problem, 
we conclude that $H_{n-1}(X;\ZZ)$ is free.

Now apply part (a) to conclude
that  $X(n-1)$ is an $(n-1)$--dimensional $(n-1)$--skeleton
and $X(n)$ is an $n$--skeleton of $X$.   
Since $H_n(M(G,n);\ZZ) = 0$, we have
$M_{n-1} \simeq *$,  so $X(n) \simeq X(n-1)$, and hence $X(n)$ is an 
 $(n-1)$--dimensional $n$--skeleton for $X$.
\end{proof}

\subsection[Lusternik-Schnirelmann category]{Lusternik--Schnirelmann category}

We make use of three equivalent definitions of the Lusternik--Schnirelmann
category of maps and spaces.

\begin{defn}
The Lusternik--Schnirelmann \term{category} of a map  $f \co X\to Y$
is the least integer $k$ for which $X$ has a 
cover by open sets
$$
X= X_0 \cup X_1 \cup \cdots \cup  X_k
$$
such that $f|_{X_i} \simeq *$ for each $i$.
When $f= \id_X$, we write $\cat(X)=\cat(\id_X)$
and when $i \co A\inclds X$, we write $\cat_X(A) = \cat(i)$. 
\end{defn}
If $X$ is a CW complex, then it is equivalent to require
each $X_i$ to be a subcomplex of $X$ in some CW decomposition.
 
The category of $f \co X\to Y$ can also be defined in terms of the 
\textit{Ganea fibrations}  $p_k \co G_k(Y) \to Y$ with fiber $F_k(Y)$.
The inductive definition of these fibrations begins by defining
$$\xymatrix@1{F_0(Y) \ar[r] &  G_0(Y)\ar[r]^-{p_0} & Y}$$
to be the familiar path-loop fibration sequence
$$\xymatrix@1{\om (Y) \ar[r] &  {\mathcal P}(Y)\ar[r]  & Y}.$$ 
Given the $k^{\mathrm{th}}$ Ganea fibration sequence
$$\xymatrix@1{F_k(Y) \ar[r] &  G_k(Y)\ar[r]^-{p_k} & Y},$$
let $\wwbar{G}_{k+1} (Y) = G_k(Y) \cup CF_k(Y)$
be the cofiber of $p_k$
and define $\wbar{p}_{k+1} \co \wwbar{G}_{k+1}(Y)\to Y$ by 
sending the cone to the
base point of $Y$.  The $(k+1)^{\mathrm{st}}$ Ganea fibration 
$p_{k+1} \co G_{k+1}(Y)\to Y$
results from converting the map $\wbar {p}_{k+1}$ to a fibration.
A result of  Ganea \cite{Ganea} 
implies that $\cat(f) \leq k$ if and only if there is a lift
$\lambda$ of $f$ in the diagram
$$
\xymatrix{
&& G_k(Y) \ar[d]^{p_k}
\\
X\ar[rr]^f \ar@{..>}[rru]^\lambda && Y.
}
$$ 
Our third definition is due to G\,W Whitehead.
According to \cite[page 458]{Whitehead},
 $\cat(f) \leq k$ if and only if the composition
 of $f$ with the diagonal map   of pairs
$$
\xymatrix{
  (X,*) \ar[r]^-f & (Y,*)\ar[r]^-{\Delta_{k+1}} & (Y,*)^{k+1}
}
$$
factors, up to homotopy of pairs,  through 
the trivial pair $(X,X)$.

We will make use of a related invariant, called $\Qcat$,
which is defined 
in terms of the fibrations that result from applying a fiberwise
version of the infinite suspension functor $Q$ to the Ganea fibrations.
Let $q_k \co \wwtilde{G}_k(Y) \to Y$
denote  the fiberwise infinite suspension of the $k$--th Ganea
fibration.  Then   $\Qcat(f)$ is the least integer $k$ for which
$f$ lifts through $q_k$ (see Scheerer, Stanley and Tanr\'e \cite{SST}).

\subsection[Rational homotopy and Lusternik-Schnirelmann category]{Rational homotopy and Lusternik--Schnirelmann category}

We briefly recall some key elements of the rational 
theory of Lusternik--Schnirelmann category.  The reader is encouraged to 
consult Cornea, Lupton, Oprea and Tanr\'e \cite[Chapter 5]{CLOT} or 
F\'elix, Halperin and Thomas \cite[Part V]{FHT0} for details.

A (simply-connected) \term{Sullivan algebra} is a 
commutative differential graded
algebra  (CDGA)  $\A$ over $\QQ$ such that: (a) $\A^0 \cong \QQ$ and
$\A^1 = 0$; (b) as a $\QQ$--algebra,  $\A \cong \Lambda(V)$ where $V$
is a graded  vector space; and (c) the differential $d$ 
is decomposable in the sense that 
  $d(\A) \sseq \bar\A^2$, where $\bar \A$ is 
the augmentation ideal of $\A$.
Every simply-connected space $X$ has a \term{Sullivan minimal model}, 
$\M(X)$, which is a Sullivan algebra such that 
$H^*(\M(X))\cong H^*(X;\QQ)$.  A \term{model} for $X$ 
is any CDGA  for which there is a map $\phi \co \M(X) \to \A$
which induces an isomorphism in cohomology 
($\phi$  is a \term{quasi-isomorphism}). 

\begin{defn}
Let $\A$ be an augmented CDGA and write $\bar \A$ for 
the augmentation ideal.  The \term{nilpotency} of $A$,  
denoted $\nil(\A)$, is the
greatest integer $k$ such that $(\bar \A)^k\neq 0$.
\end{defn}

The algebraic study of the Lusternik--Schnirelmann category of rational
spaces can be developed from the following result,
which can be found in \cite[Corollary 5.16]{CLOT}
(though, historically, it was not \cite[Remark 5.15]{CLOT}).

\begin{thm}\label{thm:Qcat}
If $X$ is a rational space, then the following are equivalent
\begin{problist}
\item  
$\cat(X) \leq k$, and 
\item
$\M(X)$ is a retract (up to chain homotopy) of a Sullivan algebra  $\B$ which is 
quasi-isomorphic with another  CDGA  $\A$ with $\nil(\A)\leq k$.
\end{problist}
\end{thm}

It follows immediately from \fullref{thm:Qcat}   that
if $u\in H^*(Y) = H^*(\M(Y))$ can be represented by a 
cocycle which is a $k$--fold product, then $f^*(u) = 0$
for any map $f \co X\to Y$ with $\cat(X) < k$.   In this case,
the \term{(rational) category weight} of $u$ is at least $k$.
We write $\wgt(u) \geq k$ and observe that $\cat(X) \geq \wgt(u)$
whenever $u\neq 0\in H^*(X)$.
 The maximum value of $\wgt(u)$ for $u\in H^*(X)$ is 
known as the \term{Toomer invariant} of $X$, and is denoted $e_0(X)$.

There is a related invariant, denoted $\Mcat$  (see F\'elix \cite{Felix}).  It is known
 that   $\Mcat(X) = \cat(X)$ for simply-connected rational spaces
(see Hess \cite[Theorem 0]{Hess}).
The equality of $\Mcat$ and $\cat$ is known to fail for 
maps:  according to Parent 
\cite[Theorems 2 and 11]{Parent} 
$\Mcat(f\cross g) = \Mcat(f) + \Mcat(g)$;
on  the other hand, Stanley \cite{Stanley} has produced examples 
of maps $f$ and $g$ between simply-connected rational spaces 
such that $\cat(f\cross g) < \cat(f) + \cat(g)$.
It is also known that $\Mcat(X) 
= \Qcat(X)$ when $X$ is a simply-connected rational space (see Scheerer
and Stelzer \cite{SS}, but see also \cite[Theorem 5.49]{CLOT}).
A simple adaptation of the proof of \cite[Theorem 5.49]{CLOT}
yields the following generalization to maps; we omit the proof.  

\begin{prop}\label{prop:Q=M}
If $f \co X\to Y$  is a map between simply-connected 
rational spaces,  then  $\Qcat(f) = \Mcat(f)$.
\end{prop}

\section{Categorical sequences}

In this section we will define our object of study,
the categorical sequence associated to a space $X$.
To ensure that our sequences are well-defined, we
must first prove some   results concerning
the relative category of an $n$--skeleton.

\subsection{Relative category of skeleta}

Since we usually think of an $n$--skeleton as a subspace of $X$,
we will sometimes write $\cat_X(X_n)$ instead of $\cat(i)$ when  
$i \co X_n\to X$ is an $n$--skeleton.

\begin{prop}\label{prop:invariant}
For fixed $n$, the integer $\cat_X(X_n)$ depends only on the 
homotopy type of $X$, and not on the choice of $n$--skeleton. 
\end{prop}

\begin{proof}
Let $i \co A\to X$ and $j \co B\to X$ be two $n$--skeleta of $X$
and consider the diagram 
$$
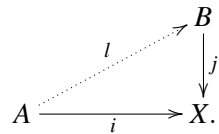
\xymatrix{
&& B\ar[d]^j
\\
A \ar[rr]_i \ar@{..>}[rru]^l && X .
}
$$ 
Since $j$ is an $n$--equivalence and $A$ is $n$--dimensional,
there is a lift $l \co A\to B$ such that $j\of l \simeq i$ 
(see Switzer \cite[Theorem 6.31]{Switzer}).  It follows 
that $\cat_X(A) = \cat(i) \leq\cat(j)$ (see Berstein and Ganea
\cite[1.4]{B-G}).  
Since the situation is symmetrical, we also have $\cat(j) \leq \cat(i)$.
\end{proof}

It can be conceptually easier to work with the Lusternik--Schnirelmann 
category of spaces
rather than   of maps.  Happily,
 there is no difference between the two for 
skeleta.

\begin{prop}\label{prop:WellDefined}
If $X$ is $(c-1)$--connected and $i \co X_n \to X$ is an $n$--skeleton
with $n\geq c$, then
\begin{problist}
\item
$\cat  (X_n) = \cat(i)$,  
\item
$\Qcat(X_n) = \Qcat(i)$, and 
\item
if $X$ is a rational space and $i \co X_n\to X$ is a rational $n$--skeleton
for $X$,  then $\cat(X_n) = \Mcat(i)$.
\end{problist}
\end{prop}

\begin{proof}
We begin by proving (a).
It is trivial that $\cat_X(X_n) \leq \cat(X_n)$;  we wish to 
prove the reverse inequality.  Assume that $\cat_X(X_n) = k$;
we will show that $\cat(X_n) \leq k$.  
Since $n\geq c$, the map $i_* \co \pi_c(X_n) \to \pi_c(X)$
is nontrivial, and hence $k \geq 1$.
Now consider the diagram
$$
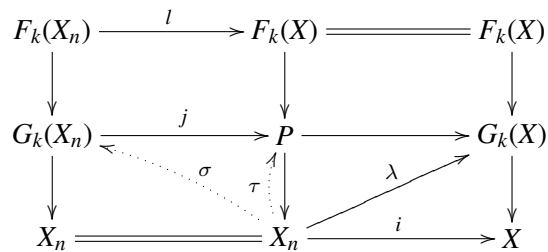
\xymatrix{
F_k(X_n) \ar[rr]^l\ar[d] && F_k(X) \ar@{=}[rr]\ar[d] && F_k(X) \ar[d]
\\
G_k(X_n) \ar[d]\ar[rr]^-j && P \ar[d] \ar[rr] && G_k(X) \ar[d]
\\
X_n \ar@{=}[rr] && X_n \ar[rru]^\lambda\ar[rr]^i \ar@{..>}@/^/[u]^\tau
\ar@{..>}@/_/[llu]_(.4){\sigma}&& X 
}
$$
in which the bottom right square is a pullback.
Since $\cat_X(X_n) = k$
there is a lift  $\lambda$ of $i$.   By the pullback property,
there is a section $\tau \co X_n \to P$.

 According to  
\cite[Lemma 6.26]{CLOT}, the map 
$l \co F_k(X_n) \to F_k(X)$  is an $(n{+}kc{-}1)$--equivalence
since $k\geq 1$,
and it follows that $j$ is also a $(n{+}kc{-}1)$--equivalence.
Since $n\leq n+ kc-1$, it follows that there is a (unique)
map $\sigma \co X_n\to G_k(X_n)$
with $j\of \sigma = \tau$ \cite[Theorem 6.31]{Switzer}.   
This $\sigma$ is a   section (up to homotopy)
of the fibration $G_k(X_n) \to X_n$,  and so $\cat(X_n)\leq k$.

The key to the proof of part (a) is the fact that 
$l \co F_k(X_n) \to F_k(X)$  is an $(n{+}kc{-}1)$--equivalence.
But this implies that  $Ql \co QF_k(X_n) \to QF_k(X)$  
is also an $(n{+}kc{-}1)$--equivalence, and so the proof of (a)
can be used again to show $\Qcat(i) = \Qcat(X_n)$.

It remains to prove (c).   For this we simply compute 
$$
\begin{array}{rclcl}
\cat(X_n) &=& \Mcat(X_n) 
&\quad & \text{by Hess \cite[Theorem 0]{Hess}}
\\
&=&  \Qcat(X_n) 
&\quad & \text{by Scheerer and Stelzer \cite{SS}}
\\
&=& \Qcat(i) 
&\quad & \text{by part (b)}
\\
&=& \Mcat(i)
&\quad & \text{by \fullref{prop:Q=M}}.
\end{array}
$$ 
This completes the proof.
\end{proof}

\begin{rem}{\rm
The proof of \fullref{prop:WellDefined}(a)
is an adaptation of  the proof 
of \cite[Theorem 1]{FHT1}.   The argument actually works
 equally well with $i \co X_n\to X$ replaced by 
any $n$--equivalence $f \co Z\to X$ with $\dim(Z) \leq n+ kc-1$.
The conclusion in this case is that $\cat(f) = \cat(Z) = k$.
}
\end{rem}

\subsection{Sequences from topology and algebra} 

We will be concerned with  sequences whose values are
either nonnegative integers or $\infty$;  thus a sequence
is a function $\sigma \co \NN\to \NN\cup \{\infty\}$.  We say
that $\sigma \leq \tau$ if $\sigma(k) \leq \tau(k)$ for each 
$k\geq 0$.  We write $\sigma <\tau$ if $\sigma \leq \tau$
and $\sigma \not  = \tau$ ($\sigma <\tau$ does \textit{not}
mean that $\sigma(k) < \tau(k)$ for every $k$).
If $\sigma$ is increasing, then the \term{length}
of $\sigma$ is $\sup\{ k\, |\, \sigma(k) < \infty\}$.

In view of Propositions \ref{prop:invariant} and \ref{prop:WellDefined}, we 
may make the following definition.

\begin{defn}
The \term{categorical sequence} of a CW complex $X$ is the 
sequence  $\sigma_X \co \NN \to \NN\cup\{\infty\}$ defined by 
$$
\sigma_X (k) = \inf\{ n \, |\,  \cat_X(X_n) \geq k\}.
$$ 
\end{defn}

\begin{rem}\label{rem:increasing} {\rm
The following elementary observations about categorical sequen\-ces
will be used frequently in what follows. 
\begin{problist}
\item
$\sigma_X$
is an  invariant of the weak homotopy type of $X$.
\item
If $X$ is $(c-1)$--connected but not $c$--connected,
then $\sigma_X(0) = 0$ and  $\sigma_X(1) = c$. 
\item
The finite values of  $\sigma_X$ are strictly increasing.
\item
If $\sigma_X(k) = n$,  then  $X_n \neq X_{n-1}$ in every cellular
decomposition of  $X$.  In particular,  if $X$ is simply-connected
and $\sigma_X(k) = n$,
then $H^n(X) \neq 0$ for some coefficients (see \fullref{thm:subadditive}(b)
below).
\item
If $X$ is  finite-dimensional, then $\cat(X) = \mathrm{length}(\sigma_X)$;
if $X$ is infinite-dimen\-sion\-al, then
$\mathrm{length}(\sigma_X)\leq 
\cat(X) \leq 2\cdot \mathrm{length}(\sigma_X )$ (see Hardie
\cite{Hardie2}).
\item
In particular, $\cat(X) = \infty$ if and only if $\mathrm{length}(\sigma_X) = \infty$. 
\item
If $\sigma_X \leq \sigma_Y$ and $Y$ is 
finite-dimensional, then $\cat(X) \geq \cat(Y)$.
\end{problist}
}
\end{rem}

Before proceeding further, we  give some examples.

\begin{authex}\label{ex:1}
(a)\qua As is well-known, the integral cohomology of 
the symplectic group $Sp(2)$
is $H^*(Sp(2)) = \Lambda(x_3,x_7)$, an exterior 
algebra on generators in dimensions $3$ and $7$.
It follows from  \fullref{thm:subadditive}(b)
that the only possible finite values for $\sigma_X(k)$
are $0,3,7$ and $10$.   Since it is known (see Schweitzer \cite[Example
4.4]{Schweitzer})
that $\cat(Sp(2)) = 3$, $\sigma_{Sp(2)}(3) <\infty$,   and hence
$$
\sigma_{Sp(2)} = (0,3,7,10,\infty,\infty, \ldots).
$$
(b)\qua Define a Sullivan algebra  $\M = \Lambda( x_3, y_3, z_5)$
with $d(z_5) = x_3 y_3$, and let $X$ be a rational space 
whose minimal model is isomorphic to $\M$ (this algebra and 
space appear in \cite[page 387]{FHT0}).  The 
nontrivial cohomology of $X$ is   
$$
\begin{array}{rcrcl}
H^3(X) &=&\QQ [x] &\oplus& \QQ [y]
\\
H^8(X) &=& \QQ [xz] & \oplus& \QQ[yz]
\\
H^{11}(X) &=& \QQ [xyz]  
\end{array}
$$
where brackets indicate cohomology classes.
Thus $\cat(X) \leq 3$, and 
since  $\cat(X) \geq \wgt([xyz]) = 3$, we have
$\cat(X) = 3$.  This forces 
$\sigma_X = (0,3,8,11,\infty,\infty,\ldots)$.

(c)\qua The  `finite-dimensional'  hypothesis 
in \fullref{rem:increasing}(g) cannot be removed.  Roitberg has 
shown that the cofibers $C$ of  
certain (phantom) maps $f \co \s K(\ZZ,5) \to S^4$ have the property that
$C_n$ is a suspension for all $n$, but $\cat(C) = 2$.
Thus  $\sigma_C = (0,4,\infty,\infty, \ldots) = \sigma_{S^4}$, but 
$
\cat(C) = 2 > 1 = \cat(S^4)
$.
\end{authex}

We will often abbreviate a sequence by deleting any terms
known to be  infinite.  Thus, for example, we could summarize the 
results of \fullref{ex:1}(a,b)   by writing
$\sigma_{Sp(2)} = (0,3,7,10)$ and $\sigma_X = (0,3,8,11)$.
If we were unsure of the later values of the sequence, 
we would write, for example, $\sigma_{Sp(2)} = (0,3,7,\ldots)$;
knowing that $\cat(Sp(2))\leq 3$, we might write
$\sigma_{Sp(2)} = (0,3,7,a)$, where $a = 10$ or $a = \infty$.
 
We will also make use of the algebraic 
\term{product length sequence} of a nonnegatively 
graded augmented CGA $\A$, defined  by
$$
\sigma_\A(k) = \inf\{ n \, |\,  \exists \ \mathrm{nontrivial}\ 
k\mathrm{-fold \  products\ in}\  \A^{n}\}.
$$
If each of  $P$ and $Q$ is  either a space or a graded algebra,
then it may happen that $\sigma_P = \sigma_Q$.  
If so, then we say that $P$ and $Q$
are \term{isosequential}.   

\begin{authex} (a)\qua
The spaces 
$$\overbrace{S^2 \cross \cdots \cross S^2}^{n\
\mathrm{factors}}\quad\text{and}\quad \cp^n$$
are easily seen to be isosequential.

(b)\qua It is easy to verify that  $\sigma_{\cp^\infty} = (0,2,4,6,\ldots)$;
it is even easier to check that if  $\A =H^*(\cp^\infty;\QQ)$, 
$\sigma_\A =   (0,2,4,6,\ldots)$.  Thus the space
$\cp^\infty$ and the graded algebra $H^*(\cp^\infty;\QQ)$
are isosequential.  

(c)\qua Let $X$ be the space of \fullref{ex:1}(b).  Then
$\sigma_X = (0,3,8,11)$, but $\sigma_{H^*(X)} = (0,3,11)$,
and $\sigma_{\M(X)} = (0,3,6,11)$, so $X$ is not isosequential 
with either $H^*(X)$ or $\M(X)$.  Instead, these sequences are
related by the string of strict inequalities
$\sigma_{\M(X)} < \sigma_X <   \sigma_{H^*(X)}$.
\end{authex}

\section{Inequalities between sequences}

One of our goals   is to develop techniques for computing 
categorical sequences $\sigma_X$.
As with formulas for the calculation of $\cat(X)$,  many of our results 
for sequences  
come in the form of inequalities.

\subsection{Inequalities for general spaces}

We begin by dispensing with wedges and retracts.

\begin{prop}\label{prop:wedgeretract}
Let $X$ and $Y$ be any two spaces.  Then
\begin{problist}  
\item
$\sigma_{X\wdg Y} (k) = \min\{ \sigma_X(k), \sigma_Y(k)\} $, and
\item
if $X$ is a homotopy retract of $Y$, then  $\sigma_X \geq \sigma_Y$.
\end{problist}
\end{prop}

\begin{proof}
Part (a) follows   from  the 
formula
$\cat(f\wdg g) = \max\{ \cat(f), \cat(g)\}$.
For (b), we consider the homotopy commutative diagram
$$
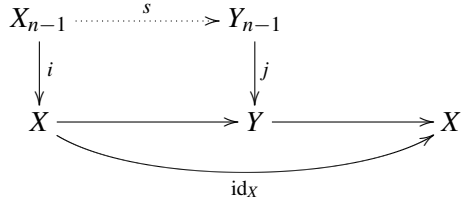
\xymatrix{
X_{n-1} \ar@{..>}[rr]^s \ar[d]^i && Y_{n-1} \ar[d]^j
\\
X\ar[rr] \ar@(rd,ld)[rrrr]_{\id_X} && Y \ar[rr] && X
}
$$
in which the map $s$ exists by cellular approximation.
It follows from the commutativity of the diagram
 that $\cat_X(X_{n-1}) = \cat(i) \leq \cat(j) = \cat_Y(Y_{n-1})$.
Now $\sigma_Y(k) = n$ implies that $\cat_Y(Y_{n-1})< k$
and hence that $\cat_X(X_{n-1}) <k$.  Therefore $\sigma_X(k) \geq n
= \sigma_Y(k)$.
\end{proof}

Our next result recasts the classical cup length lower bound for 
Lus\-ter\-nik--Schnir\-el\-mann category in terms of sequences.

\begin{prop}\label{prop:cuplength}
For any space $X$ and any ring $R$, $\sigma_X \leq \sigma_{H^*(X;R)}$.
\end{prop}

\begin{proof}
Suppose that  $\sigma_{H^*(X)}(k) = n$, so  there is a nontrivial $k$-fold 
cup product $u_1 \cdots u_k \in H^{n}(X)$.    
Let $i \co X_n \to X$ be an $n$-skeleton.  
Then $i$ induces an injection  $i^* \co H^{n}(X) \to H^{n}(X_n)$,
so  $i^*(u_1 \cdots u_k) = i^*(u_1) \cdots i^*(u_k) \neq 0\in H^n(X_n)$.
Therefore $\cat(X_n) \geq k$ (see Cornea, Lupton, Oprea and Tanr\'e
\cite[Proposition 1.5]{CLOT})
and so $\sigma_X(k) \leq n$.
\end{proof}

\fullref{prop:cuplength} can be used to determine the 
categorical sequence of a product of spheres.
This simple corollary will play an important role in our characterization
of the categorical sequences of formal rational spaces (\S5).

\begin{cor}\label{cor:WOS} 
If  $X = S^{n_1} \cross \cdots \cross S^{n_r}$ 
with $n_1\leq n_2 \leq \cdots \leq n_r$,  then 
$\sigma_X$ is given by the formula
$
\sigma_X (k) =  \sigma_{H^*(X)} (k) = n_1 + n_2 + \cdots + n_k 
$
for $1\leq k\leq r$ and $\sigma_X(k) = \infty$ for $k>r$.
\end{cor}

\begin{proof}
Clearly $\sigma_{H^*(X)} (k) = n_1 + n_2 + \cdots + n_k$, and 
\fullref{prop:cuplength} implies that 
$\sigma_X \leq   \sigma_{H^*(X)}$.
For the reverse inequality, let 
$$
X(k) = \{ (x_1, \ldots, x_r)\, |\,   
\mathrm{at\ least}\ r-k\ \mathrm{entries\ are}\  *\}  
\sseq X .
$$
It is well-known that $X(0), X(1), \ldots, X(r)$   
constitute a (spherical) 
cone decomposition of $X$.  Furthermore,
$X(k-1)$ contains the cellular $(n_1 + n_2 + \cdots + n_k-1)$--skeleton 
of $X$, 
and so 
$$
\cat(X_{n_1 + n_2 + \cdots + n_k-1}) \leq \cat(X(k-1)) < k .  
$$
Therefore 
$\sigma_X (k) \geq  n_1 + n_2 + \cdots + n_k = \sigma_{H^*(X)} (k) $.
\end{proof}

The following theorem  gives surprisingly strong algebraic control
over categorical sequences.
The proofs of parts (b) and (c) in full generality depend on 
the positive solution to Whitehead's Problem;  but they are valid
in ordinary ZFC set theory if $X$ is of finite type.

\begin{thm}\label{thm:subadditive}
For any space $X$,
\begin{problist}
\item  
$\sigma_X(k+l) \geq \sigma_X (k) + \sigma_X(l)$,  
\item  
if $X$ is simply-connected   and $\sigma_X(k) = n$, then 
$H^n(X;A)\neq 0$ for some coefficient group $A$, and 
\item  if equality occurs in (a) and $X$ is simply-connected, 
then the cup product 
$$
H^k(X;A) \otimes H^l(X;B) \to H^{k+l}(X;A\otimes B)
$$
is nontrivial for some choice of coefficients.
\end{problist}
\end{thm}

\begin{proof}
Write $\sigma_X(k) = a$ and $\sigma_X(l) = b$.  Then
$\cat(X_{a-1}) = k-1$ and $\cat(X_{b-1}) = l-1$, which means
that there are factorizations
$$
(X,*) \to (X,X_{a-1}) \to  (X,*)^k
\qquad \mathrm{and} \qquad 
(X,*) \to (X,X_{b-1}) \to (X,*)^l 
$$
of $\Delta_k$ and $\Delta_l$, up to homotopy of pairs.
Putting these together using cellular approximation and
the factorization $\Delta_{k+l} = 
(\Delta_k \cross \Delta_l)\of \Delta_2$, we obtain the homotopy-commutative
diagram  of pairs
{\footnotesize 
$$
\xymatrix{
(X_{n},*) \ar[r] \ar[d]
   & (X,*) \ar[r]^-{\Delta_2} \ar[d] 
   & (X,*)\cross (X,*) \ar[r]^-{\Delta_k \cross \Delta_l}\ar[d] 
   & (X,*)^k\cross (X,*)^l \ar@{=}[d]
\\
(X_{n},*) \ar[r] \ar[d]\ar@(ru,lu)@{..>}[rr]^(.3){\alpha}
   & (X,*) \ar[r] \ar[d]
   & (X,X_{a-1}) \cross (X,X_{b-1}) \ar[r] 
   & (X,*)^k\cross (X,*)^l
\\
(X_{n},X_{a+b-1}) \ar[r]
  & (X,X_{a+b-1}) \ar[r] 
   & (X\cross X, (X\cross X)_{a+b-1}).\ar[u] 
}
$$}
Taking $n= a+b-1$ we see that $\Delta_{k+l}|_{X_{a+b-1}}$
factors, up to homotopy of pairs, through $(X_{a+b-1},X_{a+b-1})$,
and so 
$\cat_X(X_{a+b-1}) < k+l$ by the Whitehead definition and \fullref{prop:WellDefined}.
Therefore $\sigma_X(k+l)\geq a+b$, proving (a).

Now we prove part (b).
If $\sigma_X(k) = n$, then 
$\cat_X(X_n)> \cat_X(X_{n-1})$,
so $X$ does not have an $(n{-}1)$--dimensional $n$--skeleton.
By \fullref{lem:homsection}(c), then,
it cannot be that $H^n(X;A) = 0$ for all $A$.

To prove the statement (c) about cup products,  we first recall that 
by \fullref{thm:subadditive}(b),  if $\sigma_X(i) = m$,
then $H^m(X;A)\neq 0$ for some coefficient group $A$.  
Let $u\in H^m(X;A)$ be nonzero, and interpret it as a map
$u \co X \to K(A,m)$.  This map factors
$$
\xymatrix{
X \ar[rd]^{\mu_m} \ar[d] \ar@/^/[rrd]^-u
\\
X/X_{m-1} \ar[r]_-{\kappa_m} & K(\pi_m, m) \ar[r] & K(A,m),
}
$$
where $\pi_m = \pi_m(X/X_{m-1})$.  Since $u\not\simeq *$, 
$\mu_m\not \simeq *$ as well.    Note also that $K(\pi_m,m)$
may be constructed from $X/X_{m-1}$ by attaching cells of dimension
$m+2$ and higher, so $\kappa_m$ is an $(m{+}1)$--equivalence.

Since $(X,X_{a-1}) \cross (X,X_{b-1}) = (X\cross X, 
X\cross X_{b-1}\cup X_{a-1} \cross X)$ is an $(a{+}b{-}1)$--connected
pair and $X\cross X_{b-1}\cup X_{a-1} \cross X$ is $1$--connected, 
we apply the Blakers--Massey theorem (see Switzer \cite[Corollary
6.22]{Switzer})
to conclude that the collapse
map 
$$
(X,X_{a-1}) \cross (X,X_{b-1})\to (X/X_{a-1}  \smsh  X/X_{b-1},*)
$$
is an $(a{+}b{+}1)$--equivalence.  

Assuming $\sigma_X(k+l) = \sigma_X(k) + \sigma_X(l)= a+b$,
we may set $n = a+b$ in the diagram of part (a) and conclude that
the composite map  $(X_{a+b},*) \to (X,X_{a-1}) \cross (X,X_{b-1})$
is nontrivial.  Because the collapse map is an $(a{+}b{+}1)$--equivalence
and $X_{a+b}$ is $(a{+}b)$--dimensional, we see that the composition
$$
(X_{a+b},*) \to (X,X_{a-1}) \cross (X,X_{b-1})\to  (X/X_{a-1}  \smsh  X/X_{b-1},*)
$$
is also nontrivial.  Now the desired  cup product is
$$
\xymatrix{
 X_{a+b} \ar[d]\ar@/^/[rrd]^{\mu_a\cdot \mu_b}
\\
  X/X_{a-1}  \smsh X/X_{b-1} \ar[r] & K(\pi_a,a)\smsh K(\pi_b,b)
\ar[r] & K(\pi_a\otimes \pi_b, a+b),
}
$$
and it is nontrivial because the horizontal maps are all
$(a{+}b{+}1)$--equivalences and $X_{a+b}$ is $(a{+}b)$--dimensional.
\end{proof}

The following elementary computation illustrates the use of \fullref{thm:subadditive}.

\begin{authex}
{\rm
Let us consider the exceptional Lie group $G_2$.  It is known 
(see Mimura and Toda \cite{Mimura-Toda}) that 
$
H^*(G_2;\ZZ/2) \cong \left( \ZZ/2[x_3]/(x_3^4)\right) \otimes \Lambda(x_5).
$
Therefore  
$$
\sigma_{G_2} \leq \sigma_{H^*(G_2;\ZZ/2)} = (0,3,6,9,14,\infty,\ldots) 
$$
by \fullref{prop:cuplength}.
On the other hand, 
we know $\sigma_{G_2}(1) = 3$ by \fullref{rem:increasing}(b), so 
$
\sigma_{G_2} \geq (0,3,6,9,12, \infty,\ldots) 
$
by \fullref{thm:subadditive}(b,c);  this determines 
$\sigma_{G_2}$ except for $\sigma_{G_2}(4)$.  
However, $H^*(G_2;A) = 0$ for $* = 12,13$ 
and any abelian group $A$, so $\sigma_{G_2}(4) \neq 12, 13$
by \fullref{thm:subadditive}(b).  We conclude that
$
\sigma_{G_2} =  (0,3,6,9,14).
$
}
\end{authex}

\fullref{thm:subadditive}
implies the well-known result:  
$$\cat(X) \leq {\mathrm{dimension}(X)\over \mathrm{connectivity}(X)}$$  
In \cite{Ganea2}, Ganea  generalized this familiar upper bound
to obtain an upper bound for the category of $X$ in terms of the
set of dimensions in which $H^*(X)$ is nontrivial.  We now 
prove a further generalization by a completely different method.
For a space $X$, let 
$$
h(X) = \{ n   \,  |\,  \wwtilde{H}^n(X;G) \neq 0\ \mathrm{for\ some}\ G\} .
$$ 

\begin{cor}\label{cor:Ganea}
Let $X$ be  simply-connected and of finite type with
 $\sigma_X(k) = n$.  If  there are  integers
$0<a_1< a_2<  \cdots <   a_l$ such that 
$$
h(X) \sseq  
I_1 \cup 
I_2 \cup \cdots \cup
I_l  
$$
where $I_j = [a_j, a_j + (n-1)]$
(brackets denote closed  intervals in $\RR$), 
then $\cat(X) < k (l+1)$.   
\end{cor}

\begin{proof}
Consider the integers $\sigma_X(kj)$, $j= 1, 2, \ldots$.  
We show by induction that  $\sigma_X(kj) \geq a_j$.
If $\sigma_X(kj) = \infty$
we are done, so we assume that this value  is finite,
and hence is an element of $h(X)$    
by \fullref{thm:subadditive}(b).   Since $n\in h(X)$, 
$a_1 \leq n \leq  a_1 + (n-1)$.  
Now assume that  $\sigma_X(k(j-1)) \geq  a_{j-1}$.    
By \fullref{thm:subadditive}(a),
$$
\sigma_X(kj ) \geq \sigma_X(k(j-1)) + \sigma_X(k) \geq a_{j-1} + n,
$$
which implies that    
 $\sigma_X(kj)\not \in \bigcup_{t <  j} I_t$
and forces $\sigma_X(kj) \in \bigcup_{t \geq  j} I_t \sseq [ a_j,\infty)$.

In particular, $\sigma_X(kl ) \geq  a_l$, and  so 
$\sigma_X(k(l+1) ) > \sigma_X(kl) 
+ \sigma_X(k) = a_l + n$
by \fullref{thm:subadditive}(a).  
Thus $\sigma_X(k(l+1) )\not\in h(X)$,  and so
$\sigma_X(k(l+1) )  = \infty$.
Therefore  $\cat(X) < k(l+1)$ by \fullref{rem:increasing}(e),
since the hypotheses imply that $X$ is weakly equivalent to a finite
dimensional CW complex.
\end{proof}

Ganea's theorem is the special case $k = 1$ when $X$ is $(n-1)$--connected. 
It should be noted, though, that Ganea's result applies 
for \textit{strong category} (ie, cone length), where ours 
only applies for ordinary Lusternik--Schnirelmann category.
It would be interesting to know whether our generalization 
holds with cone length in place of category.

\subsection{Rational spaces}

The categorical sequence  $\sigma_X$
for a rational space $X$ can be easily bounded above
in terms of \textit{any one} of its models.

\begin{prop}\label{prop:model}
For any simply-connected rational space $X$, and any model $\A$ for  $X$,
$
\sigma_X \geq \sigma_\A
$.
\end{prop}

\begin{proof}
Write $\sigma_\A(k) = n$ and let $\B$ be  the quotient of $\A$ by the 
differential  ideal consisting of all elements of 
dimension $n$ or greater.  Then $\nil(\B) <  k$ and the 
quotient   $q \co \A\to \B$ induces an isomorphism on cohomology
in dimensions $<n-1$ and an injection in dimension $n-1$.

Let $\N$ be the Sullivan minimal model for $\B$ and let 
$r \co \M(X) \to \N$ cover the map $q$.  Then $r$ has 
a \term{spatial realization} $i \co Z\to X$  such that  $q^* = i^* \co 
H^*(X) \to H^*(Z)$  (see F\'elix, Halperin and Thomas \cite[Chapter
17]{FHT0}).
It follows  that  $i \co Z\to X$ is a rational $(n-1)$--skeleton.
Since  $\M(Z)\sim  \B$ by construction and $\nil(\B) < k$,
we conclude using \fullref{thm:Qcat} that $\cat(Z) < k$.
It follows that $\sigma_X(k) \geq n$.
\end{proof}
 
\begin{authex}
{\rm
Let $(\A,d)$ be the CDGA with generators $x_n$, $y_m$ and $w_{n+m-1}$
(subscripts indicate dimension;  $2\leq n\leq m$)
subject to the relations $x^2 = y^2=w^2 = 0$ and with differential
determined by $dx = dy = 0$ and $dw = xy$.  This is not a Sullivan algebra,
but it does have a Sullivan model, $\M$, and $\M$ has a spatial realization,
$X$.    Then $\A$ is a model for $X$, and according to \fullref{prop:model},
$$
\sigma_X \geq \sigma_\A = (0,n,n+m, 2(n+m) -1, \infty,\infty, \ldots).
$$
But we can say even more, because the nonzero cohomology of 
$X$ occurs in dimensions $n, m, 2n+m-1, n + 2m-1$ and $2(n+m) -1$.
Since $X$ is indistinguishable from $S^n\wdg S^m$ through dimension
$n+m$, we know that $\sigma_X(2) > n+m$, and 
therefore  $\sigma_X(2) \geq 2 n+m  -1$.  Thus
$$
\sigma_X \geq  (0,n, 2n +m -1,  2(n+m) -1, \infty,\infty, \ldots).
$$
Since $\A$ is finite-dimensional, so is $H^*(X)$, and  
we  conclude that  $\cat(X) \leq 3$.    
}
\end{authex}

\section{Sequences and fibrations}

In this section we study the relationship between the sequences
$\sigma_F$, $\sigma_E$ and $\sigma_B$ 
when $F\to E\to B$ is a fibration sequence.  
Our general result is the key to a mapping theorem for 
categorical sequences of 
rational spaces.

\subsection{General spaces}
 
Our first result is proved by a slight generalization
of the method Hardie used to prove the main result of \cite{Hardie}.

\begin{prop}\label{prop:Hardie}
Consider the diagram
$$
\xymatrix{
&&  X \ar[d]_f \ar[rrd]^{p\of f}
\\
F\ar[rr]^q && E \ar[rr]^p && B
}
$$
in which the bottom row is a fibration sequence.  Then
$$
\cat(f) + 1 \leq  (\cat( p\of f) + 1)\cdot (\cat(q) + 1) .
$$
\end{prop}

\begin{proof}
Suppose $\cat(p\of f) = k$ and that $\cat(q) = l$.
Then $X$ has a cover $X = A_0 \cup A_1 \cup \cdots \cup A_k$
by subcomplexes
such that $(p\of f) |_{A_i} \simeq *$ for each $i$.  Since 
$p$ is a fibration with fiber $F$, $f|_{A_i}$ factors (up to homotopy)
as $j\of g_i$, where $g_i \co A_i \to F$.
Therefore $\cat( f|_{A_i}) \leq \cat(q) = l$ and so we 
can write  $A_i = A_{i0}\cup A_{i1} \cup \ldots \cup A_{il}$
where $(q\of g_i )|_{A_{ij}} \simeq *$.    
Thus $X = \bigcup_{i,j} A_{ij}$ where $0\leq i\leq k$ and $0\leq j\leq l$
and  $f|_{A_{ij}} \simeq *$ for all $i$ and $j$.  
Therefore $\cat(f) +1 \leq (k+1)(l+1)$.
\end{proof}

Hardie's result is the special case in which $f= \id_E$.
We are interested in the more general situation in which $f \co X\to E$
is an $n$--skeleton.

\begin{thm}\label{thm:fibseq}
Let  
$$\xymatrix@1{F\ar[r]^q & E \ar[r]^p & B}$$
be a fibration sequence and write $a= \cat(q) \leq \cat(F)$
and $b = \cat(p) \leq \cat(B)$.  Then
\begin{problist}
\item
$\sigma_E(k(a+1)  ) \geq \sigma_B(k)$, and 
\item
$\sigma_E(k(b+1)) \geq \sigma_F(k)$.
\end{problist}
\end{thm}

\begin{proof}
Let $\sigma_B(k) = n$.  Thus $\cat(B_{n-1}) < k$
and we have to show that $\cat_E(E_{n-1}) < k(a+1)$.
Consider the homotopy-commutative diagram
$$
\xymatrix{
& E_{n-1} \ar[d]_i \ar@{..>}[r] & B_{n-1}\ar[d]
\\
F\ar[r]^q & E \ar[r]^p & B ,
}
$$
in which the dotted arrow exists by cellular approximation.
According to \fullref{prop:Hardie},
$$
\begin{array}{rcl}
\cat(i)   &\leq  & (\cat(p\of i) + 1) \cdot (\cat(q) + 1) -1
\\
 &< &(\cat(B_{n-1}) + 1) \cdot (a + 1) 
\\
 &< & k(a+1),
\end{array}
$$
proving (a). 

For part (b), we let $\sigma_F(k) = n$, so $\cat(F_{n-1}) = k-1$. 
Choose 
an  $(n-1)$--skeleton $i \co E_{n-1}\to E$.  Since
$\cat(p\of i) \leq \cat(p) = b$, we can write 
$E_{n-1} = A_0 \cup A_1 \cup \cdots \cup A_b$
where $A_j$ is a subcomplex of $E_{n-1}$
(so $\dim(A_j) < n$)
and $(p\of i)|_{A_j} \simeq *$ for each $j$.
Thus $i|_{A_j}$ factors (up to homotopy) through the 
$F\to E$, and so   we have the diagram
$$
\xymatrix{
F_{n-1}\ar[d] &  A_j \ar[d]^{i |_{A_j}} \ar@{..>}[l]  
\\
F\ar[r]_q & E \ar[r]_p & B ,
}
$$
in which the dotted arrow exists by cellular approximation.
This proves that  $\cat(i|_{A_j}) \leq \cat(F_{n-1}) = k-1$,
and so $\cat(i) < (b+1)k$, which implies the desired inequality
$\sigma_E((b+1)k) \geq n$.
\end{proof}

\begin{rem}
{\rm
These inequalities are not the best possible.  A quick look at the 
proof of \fullref{thm:fibseq} shows that, in studying the 
category of $E_n$, for example, the estimate $\cat(p\of i) \leq b$
can be improved to $\cat(p\of i) \leq \cat(B_n)$, and similarly
for the second formula.   We leave the cumbersome
formulation of the sharper results to the reader.
}
\end{rem}

Since the reverse formulas expressing $\sigma_E$ in terms of 
$\sigma_B$ and $\sigma_F$ are   not entirely obvious, we record them here.

\begin{cor}
In the situation of \fullref{thm:fibseq},  
\begin{problist}
\item
$\sigma_E(k) \leq \sigma_B\bigl( \bigl\lceil {k-a \over a+1} \bigr\rceil
\bigr)$, and\vrule width 0pt depth 7pt
\item
$\sigma_E(k) \leq \sigma_F\bigl( \bigl\lceil {k-b \over b+1} \bigr\rceil
\bigr)$.
\end{problist}
\end{cor}

In \cite{Fa-Hu}, Fadell and Husseini studied the Lusternik--Schnirelmann
category of free loop spaces using a general result that relates
the category of the fiber and   the total space in a fibration sequence
with a section.   This result generalizes to a statement about 
categorical sequences.

\begin{cor}\label{cor:section}
Let  
$$\xymatrix@1{F\ar[r]  & E \ar[r]^p & B}$$
be a fibration sequence.
If $\om p$ has a section $s$, then $\sigma_E \leq \sigma_F$.
\end{cor}

\begin{proof}
Extend the given fibration sequence to the left to obtain  
$$
\xymatrix {
\om E         \ar[r]_{\om p}  
  & \om B    \ar[r]_\partial        \ar@/_/[l]_{ s}  
  & F           \ar[r] 
  & E.
}
$$
Since $\om p$ has a section, the map $\partial \co \om B \to F$ is trivial.  
Thus $\cat(\partial ) = 0$, and \fullref{thm:fibseq}(a)
implies
$
\sigma_F(k) = \sigma_F((\cat(\partial)+1) k) \geq  \sigma_E(k)
$. 
\end{proof}

We can now   expand upon the 
main homotopy-theoretical result of \cite{Fa-Hu}.

\begin{authex}
 {\rm
Let $L(X) = \map(S^1,X)$ denote the free loop space on $X$.
Evaluation at the basepoint determines a fibration $p \co L(X) \to X$
with fiber $\om X$, and the map $s \co x\mapsto l_x$, where $l_x$ is
 the constant map $l_x(S^1) = x$, is a section of $p$; 
thus $\om s$ is a section of $\om p$.
Therefore \fullref{cor:section} shows that 
$$
\sigma_{L (X)} \leq \sigma_{\om X}.
$$
In particular, $\cat(L(X)) = \infty$ if $\cat(\om X) = \infty$.
}
\end{authex}

\subsection{A mapping theorem for sequences}

One of the most powerful early results concerning the
Lusternik--Schnirelmann category of rational spaces is the 
Mapping Theorem (see F\'elix and Halperin \cite{FH});  the nice `book
proof' of this result (see F\'elix and Lemaire \cite{FL}) uses
\fullref{prop:Hardie} in the special case $\cat(j) = 0$.  
We use exactly the same argument to get an inequality for
categorical sequences.

\begin{prop}\label{prop:mapping}
Let $f \co X\to Y$ be a map between rational spaces which 
induces an injective map $f_* \co \pi_*(X)   \to \pi_*(Y)$.
Then 
$
\sigma_X \geq \sigma_Y
$.
\end{prop}

\begin{proof}
Let $q \co F\to X$ be the homotopy fiber of $f$. 
According to the proof of the standard Mapping Theorem,
the injectivity hypothesis on $f_*$ implies that 
$q\simeq *$ and so $\cat(q) = 0$ \cite[Theorem 4.11]{CLOT}.  
It now follows from 
\fullref{thm:fibseq} that $\sigma_X (k) \geq \sigma_Y(k)$
for all $k$.
\end{proof}

\section{Formal sequences}

A simply-connected space $X$ is \term{formal}
if its cohomology algebra, with trivial differential, is a model for $X$ 
\cite[page 156]{FHT0}.
In this section we characterize the categorical sequences
of formal rational spaces in several ways.

First we show that formal rational spaces  and their cohomology algebras
are isosequential.

\begin{prop}\label{prop:formal}\ 
If $X$ is a simply-connected 
formal rational space, then $\sigma_X = \sigma_{H^*(X)}$.
\end{prop}

\begin{proof}
By assumption, $H^*(X)$ is a model for $X$.
 Propositions \ref{prop:model}
and \ref{prop:cuplength} show that
$
\sigma_{H^*(X)} \leq \sigma_X \leq \sigma_{H^*(X)}
$,
which proves the result.
\end{proof}

Our main result in this section
completely characterizes the sequences which can occur as
categorical sequences of simply-connected rational formal spaces. 

\begin{thm}\label{thm:formal}
The following conditions on a sequence $\sigma$ with $\sigma(1)>1$
are equivalent:
\begin{problist}
\item 
$\sigma = \sigma_\A$ for some CGA $\A$,
\item  
$\sigma(k+1) \geq  {k+1\over k}\thinspace \sigma(k)$ for each $k$,
\item  
$\sigma = \sigma_W$ where 
 $W= \bigvee P_i$ and   $P_i = \prod S^{n_j}$ 
is a product of spheres, and
\item 
$\sigma= \sigma_X$ for some formal space $X$.
\end{problist}
\end{thm}

Before proceeding to the proof of \fullref{thm:formal} 
we need to establish a technical result about sequences.
Let $0< k \leq n$ be integers, write $n = kx + r$
with $0\leq r < k$ and let $r + s = k$.
Define $\tau$ to be the  sequence
whose finite values are
$$
\tau = ( 0, x, 2x, \ldots, sx, sx +(x+1), sx + 2(x+1), \ldots,
\underbrace {sx + r(x+1)}_n ).
$$
We call $\tau$ the  
\textit{optimal $k$--term sequence with $\tau(k) = n$}.  

\begin{lem}  \label{lem:WPS}
Assume that $\sigma$ is a sequence satisfying condition (b)
of \fullref{thm:formal}, and that $\sigma(k) < \infty$.  
Let $\tau$ be the optimal $k$--term sequence
with $\tau(k) = \sigma(k)$.  Then  $\sigma  \leq \tau $.
\end{lem}

\begin{proof}
This is clearly true for $j >k$, because $\tau(j) = \infty$ for such 
$j$.
If $\sigma(j)  > \tau(j)$ for some $j\leq k$, then   $\sigma(j)\geq \tau(j) +1$, 
and so
$$
\sigma(j+1) \geq     \tfrac1j\sigma(j)+ \sigma(j)
\geq  \tfrac1j\sigma(j)+ (\tau(j) + 1)
$$
Now 
$\sigma(j) > \tau(j) \geq jx$,
so ${1\over j}\sigma(j) > x$.
Therefore 
$$
\sigma(j+1) > \tau(j) + (x+1)  \geq \tau(j+1).
$$  
Inductively,  we see that 
$\sigma(l) > \tau(l)$ for all $i\leq l\leq k$, which contradicts the 
hypothesis $\sigma(k) = \tau(k)$.
\end{proof}

\begin{proof}[Proof of \fullref{thm:formal}]
We begin by  proving that (a) implies (b).
Let $\A$ be a CGA such that $\sigma= \sigma_\A$.
If $\sigma= (0,n)$ has length $1$, then there is nothing to prove,
so we proceed by induction, assuming that the implication is 
valid for sequences of length $\leq k$.    Write $n = \sigma(k+1) 
= \sigma_\A(k+1)$.
Then there is a nontrivial product $x_1x_2 \cdots x_{k+1} \in \A^n$,
where we write  the terms in order so that 
$|x_1| \leq |x_2|\leq \cdots \leq |x_{k+1}|$.    For $j\leq  {k+1}$ we have
$$
x_1 x_2 \cdots x_j \neq 0 \in \A^{|x_1| + |x_2|+ \cdots + |x_j|},
$$
so $\sigma_{\A}(j) \leq |x_1| + |x_2|+ \cdots + |x_j|$ for each $j$.
Since $\sigma_{\A}({k+1}) = |x_1| + |x_2|+ \cdots + |x_{k+1}|$ 
by construction, we have 
$$
\begin{array}{rcl}
\sigma_\A({k+1}) - \sigma_\A(k)  
&\geq &  (|x_1| +   \cdots + |x_{k+1}|) -(|x_1| +   \cdots + |x_{k}|)
\\
&=& |x_{k+1}|
\\
&=& {1\over {k}}
\overbrace{( |x_{k+1}| + |x_{k+1}| + \cdots + |x_{k+1}|)}^{k\ \mathrm{terms}}
\\
&\geq&  {1\over {k}}( |x_1| + |x_2|+ \cdots +|x_{k}|)
\\
&\geq& {1\over {k}} \sigma_\A(k),
\end{array}
$$
which proves the result. 

Next we prove   that  (b) implies (c)
by induction on the length
$k$ of the sequence $\sigma$.  If $\sigma = (0,n)$,
then $\sigma = \sigma_{S^n}$ and the result holds.     Suppose 
now that the result is known for all sequences with length $\leq k$,
and let $\sigma$ be a sequence with length ${k+1}$.  Write $\bar\sigma$
for the sequence 
$$
\bar\sigma (j) = \left\{ 
\begin{array}{ll}
\sigma (j) & \mathrm{if}\ j\leq  k
\\
\infty & \mathrm{if}\   j > k.
\end{array}
\right.
$$
Since $\mathrm{length} (\bar\sigma) \leq k$, we can 
apply the inductive hypothesis, to find a wedge of products
of spheres $W$ such that $\sigma_W = \bar\sigma$.    Let $\tau$ 
be the optimal $({k+1})$--term sequence
with $\tau({k+1}) = \sigma({k+1})$,  and define
$$
P = \overbrace{S^{x} \cross S^x \cross \cdots \cross S^x}^{s\ 
\mathrm{factors}}
\cross 
\overbrace{
 S^{x+1} \cross S^{x+1} \cross \cdots \cross S^{x+1}
}^{r\ 
\mathrm{factors}} .
$$ 
Then $\sigma_P = \tau$ by \fullref{cor:WOS},
and \fullref{prop:wedgeretract} shows that
$$
\sigma_{W\wdg P} (j) = \min\{ \sigma_W(j), \sigma_P(j)\}
$$
for all $j$.   For $j\leq k$, we have $\sigma_W(j) = \sigma(j) \leq \tau(j)
= \sigma_P(j)$ by \fullref{lem:WPS}, so   $\sigma_{W\wdg P} = \sigma(j)$
for $j< k$ by \fullref{prop:wedgeretract}(a).    
Also $\sigma_P({k+1}) = \sigma({k+1}) < \infty = \sigma_W({k+1})$,
so $\sigma_{W\wdg P}({k+1}) = \sigma ({k+1})$.

The implication (c) $\Rightarrow$ (d) 
follows from the fact that the  rationalization of a wedge of products
of spheres is formal.

According to \fullref{prop:formal},
if $X$ is a formal rational space, then $\sigma_X  = \sigma_{H^*(X)}$.
Thus (d) implies (a).
\end{proof}

In view of \fullref{thm:formal}, we define a 
\term{formal sequence} to be any
sequence $\sigma$ which satisfies the condition
$$
\sigma(k+1) \geq {k+1\over k} \sigma(k)\qquad \mathrm{for\ all}\ k.
$$ 

It is \textit{not} true that every formal space is isosequential 
with its minimal model.  For example,   the minimal model of 
$S^4$ is  $\Lambda(x_4,x_7)$, so 
$$
\sigma_{S^4} = (0,4) > (0,4,8,12, \ldots ) = \sigma_{\M(S^4)}.
$$

Our study of formal sequences grew out of a simple question:
is   every simply-connected rational space  isosequential
with a product of spheres, or a wedge of products of spheres, or
a product of wedges of products of spheres, etc?

Any space constructed from spheres by 
repeatedly taking products and wedges
is automatically formal \cite[Example 5.4]{CLOT}.
Using \fullref{thm:formal}, we see that any such space is isosequential with a simple
wedge of products of spheres.
Furthermore, \fullref{thm:formal} 
reveals that our  original question
reduces to asking whether or not $\sigma_X$ is a formal sequence
whenever $X$ is a rational space.    We have already seen that this is
not the case!

\begin{authex}
{\rm
The space $X$ of \fullref{ex:1}(b) is a rational space 
whose categorical sequence is $\sigma_X = (0,3,8,11)$.  
Since $11 < {3\over 2}\cdot 8$,  $\sigma_X$ is not a formal
sequence.  By  \fullref{thm:formal},
$X$ is not isosequential with any wedge of products of spheres.
}
\end{authex}

\section{Products}

For two sequences $\sigma$ and $\tau$, we 
define a new sequence $\sigma*\tau$ by
$$
\sigma*\tau (k) = \min\{ \sigma(i) + \sigma(j)\, |\, i+j = k\}.
$$
Our goal in this section is to prove a result linking the
sequences  $\sigma_{X\cross Y}$  and $\sigma_X * \sigma_Y$.
 When the spaces in question are formal, this is not hard to do.

\begin{prop}  
Let $\A$ and $\B$ be simply-connected CGAs and let 
$X$ and $Y$ be simply-connected formal rational spaces. 
Then
\begin{problist}
\item   
$\sigma_{\A\tensor \B} = \sigma_\A * \sigma_\B $, and 
\item
$\sigma_{X\cross Y} = \sigma_X * \sigma_Y$.
\end{problist}
\end{prop}

\begin{proof}
We omit the easy proof of (a), and use it to prove (b)
as follows: since $X$, $Y$ and $X\cross Y$ are each formal and rational, 
$$\sigma_{X\cross Y} =  \sigma_{H^*(X\cross Y)}
  = \sigma_{H^*(X) \otimes H^*(Y)}
  = \sigma_{H^*(X)} * \sigma_{H^*(Y)}
  =  \sigma_X * \sigma_Y$$
by \fullref{prop:formal}.
\end{proof}

The following conjecture  seems quite plausible.
 
\begin{conj}\label{conj:1}  
For simply-connected rational  $X$ and $Y$, 
$\sigma_{X\cross Y} =  \sigma_X * \sigma_Y$.
\end{conj}

Unfortunately, we have been unable to prove this.
However, we can prove that there is an inequality relating
these sequences.

\begin{thm}\label{thm:product}
For simply-connected rational  $X$ and $Y$, 
$\sigma_{X\cross Y} \leq  \sigma_X * \sigma_Y$.
\end{thm}

\begin{proof} 
Let $\sigma_X * \sigma_Y(k) = n$.  Thus there
are  $i$ and $j$ with $i+ j = k$, $\sigma_X(i) = a$,
$\sigma_Y(j) = b$,  and   $a+ b = n$.   Now let 
$i_a\cross i_b \co X_a \cross Y_b \to X\cross Y$ and compute
$$
\begin{array}{rclcl}
\cat( (X\cross Y)_{n})   &\geq&    \cat_{X\cross Y}(X_a \cross Y_b) 
\\
&\geq & \Mcat(i_a \cross i_b)   &\quad &
\text{by\  \fullref{prop:WellDefined}(c)}
\\
&=& \Mcat(i_a) + \Mcat( i_b)   
&\quad & \text{by Parent \cite[Theorem 2]{Parent}}
\\
&= & \cat(X_a) + \cat(X_b ) 
&\quad & \text{by\ \fullref{prop:WellDefined}(c)}\   
\\
&=& k,
\end{array}
$$
which means that 
$\sigma_{X\cross Y} (k)  \leq  n = \sigma_X * \sigma_Y(k)$,
\end{proof}

The inequality of \fullref{thm:product} fails when the 
spaces are not rational, as the following example demonstrates.

\begin{authex}
{\rm
Iwase \cite{Iwase} has constructed a space $X = S^2 \cup D^{10}$
with the property that $\cat(X\cross S^k) = \cat(X) = 2$
for all $k\geq 2$.    The categorical sequences for  $X$ 
and $S^2$ are  
$
\sigma_X = (0, 2, 10, \infty, \ldots)$ and 
$\sigma_{S^2} = (0,2,\infty, \ldots)$, respectively.
Now we have 
$
\sigma_X *\sigma_{S^2}  = (0, 2, 4, 12, \infty, \ldots) 
<(0,2,4,\infty,\ldots)=
\sigma_{X\cross S^2}  .
$
}
\end{authex}

Nevertheless, the following conjecture seems reasonable.

\begin{conj}\label{conj:genprod}
For general spaces $X$ and $Y$, 
$\sigma_{X\cross Y} \geq \sigma_X * \sigma_Y$.
\end{conj}

 \fullref{conj:genprod},
together with \fullref{thm:product}, 
implies  \fullref{conj:1}.

\section{The Mislin genus of $Sp(3)$}

In this final section we use categorical sequences 
to give a simple proof of a theorem of Ghienne \cite{Ghienne}.  

The \term{Mislin genus} of a nilpotent space $X$ is the set $\G(X)$
 of   homotopy types of 
 nilpotent spaces $Y$ such that the $p$--localizations
$X_{(p)}$ and $Y_{(p)}$ are homotopy equivalent
for every prime $p$.  McGibbon \cite[Section 8]{McG}
asked whether Lusternik--Schnirelmann category is an invariant
of Mislin genus;  that is, if $X\in\G(Y)$, does it follow
that $\cat(X) = \cat(Y)$?  This is known to be false for 
certain infinite-dimensional spaces (see Roitberg \cite{Roitberg}),
but the question remains open for finite complexes  $Y$.

In \cite{Ghienne}, Ghienne proved that McGibbon's conjecture 
holds in the special case $Y = Sp(3)$.  We use a   sequence
computation to give a simple alternative proof of this result.

\begin{thm}[Ghienne]\label{thm:genus}
If  $X\in \G(Sp(3))$, then $\cat(X)  = 5$.
\end{thm}

\begin{proof}
According to Fern\'andez-Su\'arez, G\'omez-Tato, Strom and Tanr\'e
\cite{FGST}, and Iwase and Mimura \cite{IM}, $\wcat(Sp(3))= \cat(Sp(3)) = 5$.
Since weak category is a genus invariant, we have 
$$
\cat(X) \geq \wcat(X)  = \wcat(Sp(3)) = 5
$$
for any space $X\in \G(Sp(3))$.  It remains to 
show that  $\cat(X) \leq 5$ for every 
$X\in \G(Sp(3))$.  
In fact, we prove the following stronger statement:
 any simply-connected space $X$ whose cohomology ring $H^*(X;\ZZ)$
is isomorphic to 
$H^*(Sp(3);\ZZ)$  must have  $\cat(X) \leq 5$.   

The categorical
sequence $\sigma_X$ clearly  has $\sigma_X(1) = 3$ and 
$\sigma_X(2) \geq 7$ by \fullref{thm:subadditive}(b).  
By \fullref{thm:subadditive}(a),
$\sigma_X(4)\geq \sigma_X(2) + \sigma_X(2) \geq 14$.
Furthermore, $\sigma_X(4) > 14$ by \fullref{thm:subadditive}(c),
because
the cup product $H^7(X) \tensor H^7(X) \to H^{14}(X)$ is trivial.
Now we have $\sigma_X(4) \geq 18$ by \fullref{thm:subadditive}(b),
and hence  $\sigma_X(5) \geq \sigma_X(4) + \sigma_X(1) = 21$.
From this we immediately conclude that  $\cat(X) = \cat(X_{21}) \leq 5$.
\end{proof}

McGibbon's conjecture for finite complexes is equivalent to
the following conjecture for finite type spaces.

\begin{conj}\label{conj:McG}
If $X$ is a nilpotent space of finite type, then
$\sigma_Y = \sigma_X$ for every $Y\in \G(X)$.
\end{conj}

 \fullref{conj:McG} is easily seen to be valid for 
$X = Sp(2)$.   
We can also verify the conjecture for  $X = Sp(3)$.

\begin{cor}
If $X\in \G(Sp(3))$, then $\sigma_X = \sigma_{Sp(3)} = (0,3,7,10,18,21)$.
\end{cor}

\begin{proof}
The proof of \fullref{thm:genus} shows that
if $H^*(X) \cong H^*(Sp(3))$ then $\sigma_X \geq (0,3,7,10,18,21)$.
If $\cat(X) = 5$, then $\sigma_X(5) \leq 21$, and this implies
$\sigma_X(2) = 7$.   Now \cite[Theorem 8]{Strom} implies that 
$\cat(X_{10}) = 3$, and hence $\sigma_X(3) = 10$. 
The analysis used in the proof of \fullref{thm:genus}
shows that  $\sigma_X = (0,3,7,10,18,21)$.
\end{proof}

\bibliographystyle{gtart}
\bibliography{link}

\begin{thebibliography}{}
\providecommand\bibmarginpar{\leavevmode\marginpar}
\def\urlstyle#1{{\tt #1}}

\bibitem{A-Z}
\textbf{A Ambrosetti}, \textbf{V Coti~Zelati},
  \href{http://dx.doi.org/10.1007/BF01773936} {\emph{Critical points with lack
  of compactness and singular dynamical systems}}, Ann. Mat. Pura Appl. $(4)$
  149 (1987) 237--259 \xox{MR}{932787}

\bibitem{B-G}
\textbf{I Berstein}, \textbf{T Ganea}, \emph{The category of a map and of a
  cohomology class}, Fund. Math. 50 (1961/1962) 265--279 \xox{MR}{0139168}

\bibitem{B-H}
\textbf{I Berstein}, \textbf{P\,J Hilton}, \emph{Category and generalized
  {H}opf invariants}, Illinois J. Math. 4 (1960) 437--451 \xox{MR}{0126276}

\bibitem{BrownCopeland}
\textbf{E\,H Brown, Jr}, \textbf{A\,H Copeland, Jr}, \emph{An homology analogue
  of {P}ostnikov systems}, Michigan Math. J 6 (1959) 313--330 \xox{MR}{0110096}

\bibitem{Dude1}
\textbf{G Cicorta{\c{s}}}, \emph{Categorical sequences and applications},
  Studia Univ. Babe\c s-Bolyai Math. 47 (2002) 31--39 \xox{MR}{1989588}

\bibitem{Dude2}
\textbf{G Cicorta{\c{s}}}, \emph{Relatively and {$G$}--categorical sequences
  and applications}, from: ``Proceedings of ``BOLYAI 200'' International
  Conference on Geometry and Topology'', Cluj Univ. Press, Cluj-Napoca (2003)
  67--74 \xox{MR}{2112613}

\bibitem{CLOT}
\textbf{O Cornea}, \textbf{G Lupton}, \textbf{J Oprea}, \textbf{D Tanr{\'e}},
  \emph{Lusternik--{S}chnirelmann category}, Mathematical Surveys and
  Monographs 103, American Mathematical Society, Providence, RI (2003)
  \xox{MR}{1990857}

\bibitem{Fa-Hu}
\textbf{E Fadell}, \textbf{S Husseini},
  \href{http://links.jstor.org/sici?sici=0002-9939(198910)107:2%3C527:ANOTCO%3%
E2.0.CO%3B2-K} {\emph{A note on the category of the free loop space}}, Proc.
  Amer. Math. Soc. 107 (1989) 527--536 \xox{MR}{984789}

\bibitem{Felix}
\textbf{Y F{\'e}lix}, \emph{La dichotomie elliptique-hyperbolique en homotopie
  rationnelle}, Ast\'erisque  (1989) 187 \xox{MR}{1035582}

\bibitem{FH}
\textbf{Y F{\'e}lix}, \textbf{S Halperin},
  \href{http://links.jstor.org/sici?sici=0002-9947(198209)273:1%3C1:RLCAIA%3E2%
.0.CO%3B2-9} {\emph{Rational {LS} category and its applications}}, Trans. Amer.
  Math. Soc. 273 (1982) 1--38 \xox{MR}{664027}

\bibitem{FHT0}
\textbf{Y F{\'e}lix}, \textbf{S Halperin}, \textbf{J-C Thomas}, \emph{Rational
  homotopy theory}, Graduate Texts in Mathematics 205, Springer, New York
  (2001) \xox{MR}{1802847}

\bibitem{FHT1}
\textbf{Y Felix}, \textbf{S Halperin}, \textbf{J-C Thomas},
  \href{http://dx.doi.org/10.1016/S0166-8641(01)00288-7}
  {\emph{Lusternik--{S}chnirelmann category of skeleta}}, Topology Appl. 125
  (2002) 357--361 \xox{MR}{1933583}

\bibitem{FL}
\textbf{Y F{\'e}lix}, \textbf{J-M Lemaire},
  \href{http://dx.doi.org/10.1016/0040-9383(85)90043-6} {\emph{On the mapping
  theorem for {L}usternik--{S}chnirelmann category}}, Topology 24 (1985) 41--43
  \xox{MR}{790674}

\bibitem{FGST}
\textbf{L Fern{\'a}ndez-Su{\'a}rez}, \textbf{A G{\'o}mez-Tato}, \textbf{J
  Strom}, \textbf{D Tanr{\'e}},
  \href{http://dx.doi.org/10.1090/S0002-9939-03-07019-9} {\emph{The
  {L}usternik--{S}chnirelmann category of {$\rm Sp(3)$}}}, Proc. Amer. Math.
  Soc. 132 (2004) 587--595 \xox{MR}{2022385}

\bibitem{Fox}
\textbf{R\,H Fox},
  \href{http://links.jstor.org/sici?sici=0003-486X(194104)2:42:2%3C333:OTLC%3E%
2.0.CO%3B2-V} {\emph{On the {L}usternik--{S}chnirelmann category}}, Ann. of
  Math. $(2)$ 42 (1941) 333--370 \xox{MR}{0004108} \xox{JFM}{0027.43104}

\bibitem{Ganea2}
\textbf{T Ganea}, \emph{Upper estimates for the {L}justernik--{S}nirel'man
  category}, Soviet Math. Dokl. 2 (1961) 180--183 \xox{MR}{0132546}

\bibitem{Ganea}
\textbf{T Ganea}, \emph{A generalization of the homology and homotopy
  suspension}, Comment. Math. Helv. 39 (1965) 295--322 \xox{MR}{0179791}

\bibitem{Ghienne}
\textbf{P Ghienne}, \emph{The {L}usternik--{S}chnirelmann category of spaces in
  the {M}islin genus of {$\rm Sp(3)$}}, from: ``Lusternik--Schnirelmann
  category and related topics (South Hadley, MA, 2001)'', Contemp. Math. 316,
  Amer. Math. Soc., Providence, RI (2002)  121--126 \xox{MR}{1962158}

\bibitem{Hardie}
\textbf{K\,A Hardie},
  \href{http://projecteuclid.org/getRecord?id=euclid.mmj/1029000521} {\emph{A
  note on fibrations and category}}, Michigan Math. J. 17 (1970) 351--352
  \xox{MR}{0283795}

\bibitem{Hardie2}
\textbf{K\,A Hardie}, \emph{On the category of the double mapping cylinder},
  T\^ohoku Math. J. $(2)$ 25 (1973) 355--358 \xox{MR}{0370558}

\bibitem{Hess}
\textbf{K\,P Hess}, \href{http://dx.doi.org/10.1016/0040-9383(91)90006-P}
  {\emph{A proof of {G}anea's conjecture for rational spaces}}, Topology 30
  (1991) 205--214 \xox{MR}{1098914}

\bibitem{Iwase}
\textbf{N Iwase}, \href{http://dx.doi.org/10.1112/S0024609398004548}
  {\emph{Ganea's conjecture on {L}usternik--{S}chnirelmann category}}, Bull.
  London Math. Soc. 30 (1998) 623--634 \xox{MR}{1642747}

\bibitem{IM}
\textbf{N Iwase}, \textbf{M Mimura}, \emph{L--{S} categories of
  simply-connected compact simple {L}ie groups of low rank}, from:
  ``Categorical decomposition techniques in algebraic topology (Isle of Skye,
  2001)'', Progr. Math. 215, Birkh\"auser, Basel (2004)  199--212
  \xox{MR}{2039767}

\bibitem{McG}
\textbf{C\,A McGibbon}, \emph{The {M}islin genus of a space}, from: ``The
  Hilton Symposium 1993 (Montreal, PQ)'', CRM Proc. Lecture Notes 6, Amer.
  Math. Soc., Providence, RI (1994)  75--102 \xox{MR}{1290585}

\bibitem{Mimura-Toda}
\textbf{M Mimura}, \textbf{H Toda}, \emph{Topology of {L}ie groups. {I}, {II}},
  Translations of Mathematical Monographs 91, American Mathematical Society,
  Providence, RI (1991) \xox{MR}{1122592}

\bibitem{Parent}
\textbf{P-E Parent}, \href{http://dx.doi.org/10.1016/S0166-8641(99)00072-3}
  {\emph{L{S} category: product formulas}}, Topology Appl. 106 (2000) 35--47
  \xox{MR}{1769330}

\bibitem{Rabinowitz}
\textbf{P\,H Rabinowitz}, \emph{Periodic solutions for some forced singular
  {H}amiltonian systems}, from: ``Analysis, et cetera'', Academic Press, Boston
  (1990)  521--544 \xox{MR}{1039360}

\bibitem{Roitberg}
\textbf{J Roitberg}, \href{http://dx.doi.org/10.1016/S0040-9383(98)00060-3}
  {\emph{The {L}usternik--{S}chnirelmann category of certain infinite
  {CW}--complexes}}, Topology 39 (2000) 95--101 \xox{MR}{1710994}

\bibitem{SST}
\textbf{H Scheerer}, \textbf{D Stanley}, \textbf{D Tanr{\'e}}, \emph{Fibrewise
  construction applied to {L}usternik--{S}chnirelmann category}, Israel J.
  Math. 131 (2002) 333--359 \xox{MR}{1942316}

\bibitem{SS}
\textbf{H Scheerer}, \textbf{M Stelzer}, \emph{Fibrewise infinite symmetric
  products and {$M$}--category}, Bull. Korean Math. Soc. 36 (1999) 671--682
  \xox{MR}{1736612}

\bibitem{Schweitzer}
\textbf{P\,A Schweitzer}, \href{http://dx.doi.org/10.1016/0040-9383(65)90002-9}
  {\emph{Secondary cohomology operations induced by the diagonal mapping}},
  Topology 3 (1965) 337--355 \xox{MR}{0182969}

\bibitem{Shelah}
\textbf{S Shelah}, \emph{Infinite abelian groups, {W}hitehead problem and some
  constructions}, Israel J. Math. 18 (1974) 243--256 \xox{MR}{0357114}

\bibitem{Stanley}
\textbf{D Stanley},
  \href{http://journals.cms.math.ca/ams/ams-redirect.php?Journal=CJM&Volume=54%
&FirstPage=608} {\emph{On the {L}usternik--{S}chnirelmann category of maps}},
  Canad. J. Math. 54 (2002) 608--633 \xox{MR}{1900766}

\bibitem{Strom}
\textbf{J Strom}, \href{http://dx.doi.org/10.1016/S0040-9383(01)00022-2}
  {\emph{Decomposition of the diagonal map}}, Topology 42 (2003) 349--364
  \xox{MR}{1941439}

\bibitem{Switzer}
\textbf{R\,M Switzer}, \emph{Algebraic topology -- homotopy and homology},
  Classics in Mathematics, Springer, Berlin (2002) \xox{MR}{1886843}

\bibitem{Whitehead}
\textbf{G\,W Whitehead}, \emph{Elements of homotopy theory}, Graduate Texts in
  Mathematics 61, Springer, New York (1978) \xox{MR}{516508}

\end{thebibliography}

\end{document}